\documentclass{amsart}
\usepackage{amssymb,latexsym}
\theoremstyle{plain}
\newtheorem{theorem}{Theorem}
\newtheorem{corollary}{Corollary}

\newtheorem{proposition}{Proposition}
\newtheorem*{TA}{Theorem}
\theoremstyle{definition}

\theoremstyle{remark}

\numberwithin{equation}{section}

\newcommand{\R}{\mathbb R}

\newcommand{\Z}{\mathbb Z}
\newcommand{\C}{\mathbb C}

\newcommand{\Rn}{\mathbb R^n}

\begin{document}

\title[Unique continuation]{Unique Continuation for 
Schr\"odinger Evolutions, with applications to profiles of concentration and traveling waves}
\author{L. Escauriaza}
\address[L. Escauriaza]{UPV/EHU\\Dpto. de Matem\'aticas\\Apto. 644, 48080
Bilbao, Spain.}
\email{luis.escauriaza@ehu.es}
\thanks{The first and fourth authors are supported  by MEC grant,
MTM2004-03029, the second and third authors by NSF grants DMS-0456583 and
DMS-0456833 respectively}
\author{C. E. Kenig}
\address[C. E. Kenig]{Department of Mathematics\\University of
Chicago\\Chicago, Il. 60637 \\USA.}
\email{cek@math.uchicago.edu}
\author{G. Ponce}
\address[G. Ponce]{Department of Mathematics\\
University of California\\
Santa Barbara, CA 93106\\
USA.}
\email{ponce@math.ucsb.edu}
\author{L. Vega}
\address[L. Vega]{UPV/EHU\\Dpto. de Matem\'aticas\\Apto. 644, 48080
Bilbao, Spain.}
\email{luis.vega@ehu.es}
\begin{abstract}
We prove unique continuation properties for solutions of the evolution Schr\"odinger equation with 
time dependent potentials. As an application of our method we also obtain results concerning 
the possible  concentration profiles of   blow up solutions 
and the possible profiles of the traveling waves solutions of semi-linear Schr\"odinger equations. 
\end{abstract}
\maketitle
\begin{section}{Introduction}\label{S: Introduction}
In this paper we continue our study initiated in \cite{ekpv06}, 
\cite{ekpv08}, \cite{ekpv08b}, and \cite{ekpv09} on unique continuation properties of
solutions of Schr\"odinger equations. To begin  with we consider the linear equation
\begin{equation}
\label{e1}
 \partial_t u =i( \Delta u + V(x,t) u),\;\;\;\;\;\; (x,t)\in \R^n\times [0,\infty).
\end{equation}

We shall be  interested in finding the 
strongest possible space decay of global solutions of \eqref{e1}. In this direction our first results are the following ones:

\begin{theorem}
\label{theorem4}

Let $u\in C([0,\infty) :L^2(\Rn))$ be a solution of the equation \eqref{e1}
with a real potential $V\in L^{\infty}(\R^{n}\times [0,\infty))$ satisfying
that
\begin{equation}
\label{split-pot}
V(x,t)=V_1(x,t)+V_2(x,t),
\end{equation}
with $V_j,\,j=1,2 $ real valued,
\begin{equation}
\label{potcon2bc}
|V_1(x,t)|\leq \frac{c_1}{\langle x\rangle^{\alpha}}= \frac{c_1}{(1+|x|^2)^{\alpha/2}},\;\;\;\;\;\;\;0\leq \alpha<1/2,
\end{equation}
and $V_2$ supported in $\{(x,t)\,:\,|x|\geq 1\}$ such that
\begin{equation}
\label{potcon2bd}
-(\partial_rV_2(x,t))^{-}\leq \frac{c_2}{|x|^{2\alpha}},\;\;\;\;\;\;\;\;\;\;a^{-}=min\{a;0\}.
\end{equation}
Then there exists a constant $\lambda_0=\lambda_0(\|V\|_{L^{\infty}(\R^{n}\times [0,\infty))};c_1;c_2;\alpha)>0$
such that if
\begin{equation}
\label{preserved-2bb}
\sup_{t\geq 0}\;\int_{\R^n}\,e^{\lambda_0\,|x|^{p}} \,|u(x,t)|^2\,dx < \infty,\;\;\;\;\;\;\text{with}\;\;\;\;\;\;\;p=(4-2\alpha)/3,
\end{equation}
then
\begin{equation}
\label{conclu3bc}
u\equiv 0.
\end{equation}
\end{theorem}

 As an immediate consequence of Theorem \ref{theorem4} we have:
 
 \begin{corollary}
\label{corollary13}

Let $u\in C([0,\infty):L^2(\Rn))$ be a solution of the equation \eqref{e1} with a real potential $V\in L^{\infty}(\R^{n}\times [0,\infty))$. If
\begin{equation}
\label{13a}
|V(x,t)|\leq \frac{c_1}{\langle x\rangle^{\alpha}}= \frac{c_1}{(1+|x|^2)^{1/2}},
\end{equation}
and   for some $p>1$ and $\lambda_0>0$ 
 \begin{equation}
 \label{13b}
sup_{t\geq 0}\;\int_{\R^n}\,e^{\lambda_0\,|x|^{p}} \,|u(x,t)|^2\,dx <\infty,
 \end{equation}  
 then $u\equiv 0$.

\end{corollary}

\begin{theorem}
\label{theorem4b}

Let $u\in C([0,\infty):L^2(\Rn))$ be a solution of the equation \eqref{e1}
with a real potential $V\in L^{\infty}(\R^{n}\times [0,\infty))$ satisfying
that
\begin{equation}
\label{split-potb}
V(x,t)=V_1(x,t)+V_2(x,t),
\end{equation}
with $V_j,\,j=1,2 $ real valued,
\begin{equation}
\label{potcon2bcd}
|V_1(x,t)|\leq \frac{c_1}{\langle x\rangle^{1/2+\epsilon_0}}= \frac{c_1}{(1+|x|^2)^{1/4+\epsilon_0/2}},\;\;\;\;\;\epsilon_0>0,
\end{equation}
and $V_2$ supported in $\{(x,t)\,:\,|x|\geq 1\}$ such that
\begin{equation}
\label{potcon2bde}
-(\partial_rV_2(x,t))^{-}\leq \frac{c_2}{|x|^{1+\epsilon_0}},\;\;\;\;\;\;\;\;\;\;a^{-}=min\{a;0\}.
\end{equation}
Then there exists a constant $\lambda_0=\lambda_0(\|V\|_{L^{\infty}(\R^{n}\times [0,\infty))};c_1;c_2;\epsilon_0)>0$
such that if
\begin{equation}
\label{preserved-2bbb}
\sup_{t\geq 0}\;\int_{\R^n}\,e^{\lambda_0\,|x|} \,|u(x,t)|^2\,dx < \infty,
\end{equation}
then
\begin{equation}
\label{conclu3bcd}
u\equiv 0.
\end{equation}
\end{theorem}

\vskip.1in
Using  the results in \cite{ekpv08} and  \cite{ekpv08b} one sees that it suffices to assume that the hypothesis \eqref{preserved-2bb} and \eqref{preserved-2bbb}
in Theorem \ref{theorem4} and Theorem \ref{theorem4b} respectively, hold for a sequence of times 
$\{\widetilde T_j=T_0+j\,L\,:\,j\in\Z^+\}$ for some $T_0\geq 0$ and $L>0$.

The hypothesis on the real character on the potential in  these theorems is used to guarantee that the $L^2$-norm of the solution of the equation \eqref{e1} is time independent. However, it suffices to have  the  $L^2$-norm of the solution  bounded below for all time $t\in [0,\infty)$ by a positive constant, provided that $u(0)\neq 0$. Therefore, Theorem \ref{theorem4} still holds for potentials $V(x,t)$ which can be written as
$$
V(x,t)=V_1(x,t)+V_2(x,t)+V_3(x,t),
$$
with $V_1$ and $V_2$ as before  and $V_3$ complex valued satisfying  \eqref{potcon2bc} and such that
$$
\|V_3\|_{L^1([0,\infty) : L^{\infty}(\R^n))}= \int_0^{\infty} \|V_3(\cdot, t)\|_{\infty}dt <\infty.
$$
A similar remark applies to Theorem \ref{theorem4b}.

Next, we define  the \lq\lq hyperbolic" or \lq\lq ultra-hyperbolic" operator
\begin{equation}
\label{opel}
\mathcal L_k = \partial_{x_1}^2+...+\partial_{x_k}^2-\partial_{x_{k+1}}^2-...-\partial_{x_n}^2,\;\;\;\;\;\;\;k\in\{2,..,n-1\},
\end{equation}
and study  the linear dispersive equation
\begin{equation}
\label{e1b}
 \partial_t u =i( \mathcal L_k u + V(x,t) u),\;\;\;\;\;\; (x,t)\in \R^n\times \R.
\end{equation}

Nonlinear models with a non-degenerate non-elliptic operator 
$\mathcal L_k$ describing the dispersive relation arise in several mathematical and physical contexts. For example, the Davey-Stewarson system \cite{DaSt} 
\begin{equation}
\label{ds}
\begin{cases}
i\partial_t u \pm\partial_x^2 u + \partial_y^2 u = c_{_1}|u|^2 u +
c_{_2}u \partial_x \varphi, \qquad t, \, x,\,y \in \R, \\
\partial_x^2 \varphi \pm\partial_y^2 \varphi = \partial_x |u|^2, 
\end{cases}
\end{equation}
with $u = u(x,y,t)$ a complex-valued function, $\varphi=\varphi(x,y,t)$
a real-valued function and
$c_{1},\,c_2$ real parameters. The system \eqref{ds} appears 
as a model in wave propagations \cite{DaSt}  and independently as a two
dimensional completely integrable system which generalizes the integrable cubic 1-dimensional Schr\"odinger equation
\cite{AbHa}.  
Also one has the Ishimori system  \cite{Is} 
\begin{equation}
\label{Ishi}
\begin{cases}
\begin{aligned}
&\partial_t S= S \wedge (\partial_x^2 S\pm \partial_y^2 S) + 
b(\partial_x\phi \partial_y S + \partial_y\phi \partial_x S),
\qquad t, \,x, \,y\in\mathbb R, \\
&\partial_x^2\phi \mp \partial_y^2\phi= \mp2S\cdot(\partial_x S\wedge\partial_y S),  
\end{aligned}
\end{cases}
\end{equation}
where $\;S(\cdot,t): \mathbb R^2\to \mathbb R^3\;$ with $\;\|S\|=1$, $\;S\to(0,0,1)\;$ as 
$\;\|(x,y)\|\to\infty$, and $\;\wedge\;$ denotes the wedge product in $\;\mathbb R^3$. 
 This model was first proposed
 as a two dimensional generalization of the
Heisenberg equation in ferromagnetism.
 For $\;b=1\;$ the system \eqref{Ishi} has been shown to be completely 
 integrable  
(see \cite{AbHa} and references therein).

 The arguments used in the  proofs of Theorems \ref{theorem4}-\ref{theorem4b}
 do not rely on the elliptic character of the laplacian in \eqref{e1}, so we have:
 
\begin{theorem}
\label{theorem2}

Theorems \ref{theorem4}-\ref{theorem4b} and Corollary \ref{corollary13} still hold for solutions $u\in C([0,\infty) :L^2(\Rn))$ of the equation  \eqref{e1b}
with a  potential $V$ verifying the same hypotheses.
\end{theorem}

\underline{Remarks} (i) It is interesting to relate our results with those  due to  V. Z. Meshkov 
 in   \cite{Mes}:
   
   \begin{TA} 
\label{theoremA} Let  $\,w\in H^2_{loc}(\Rn)$ be a solution of 
\begin{equation}
\label{estaeq}
\Delta w+\widetilde  V(x)w = 0,\;\;\;\;x\in\Rn,\;\;\;\text{with} \;\;\;\widetilde V\in L^{\infty}(\R^n). 
\end{equation}

\begin{equation}
\label{mescon}
\text{If}\;\;\;\;\;\int  \,e^{2a|x|^{4/3}}\,|w|^2 dx <\infty, \;\;\;\forall \,a>0,\;\;\text{then}\;\;\;\;w\equiv 0.
\end{equation}
\end{TA}

It was also proved in \cite{Mes} that for complex valued potentials $\widetilde V$ the exponent $4/3$ in \eqref{mescon} is
optimal. 

 We observe that if the potential in \eqref{e1} $V(x,t)$ is time independent $V=\widetilde V(x)$, then  a solution of $w(x)$  of \eqref{estaeq} 
 is a  stationary solutions of the IVP \eqref{e1}. Also for time independent  potential  $V(x,t)=\widetilde V(x)$, if $w(x)$ 
 is an $H^1$-solution of the eigenvalue problem
\begin{equation}
\label{eigen}
\Delta w + \widetilde V(x) w=\zeta w,
\end{equation}
then one has that for $\zeta \in\R$ 
\begin{equation}
\label{ground}
v(x,t)= e^{i \zeta t}\,w(x),
\end{equation}
is a  solution of the IVP \eqref{e1} for which Theorems \ref{theorem4}-\ref{theorem4b} apply. 
As it was mentioned above the assumption on the real character on the potential in  these theorems is only required to guarantee that the $L^2$-norm of the solution of the equation \eqref{e1} is time independent.  In the case described in \eqref{eigen}-\eqref{ground} the solution $v(x,t)$ preserves the $L^2$-norm and so the proof of 
Theorems \ref{theorem4}-\ref{theorem4b} can be carried out. 
Hence, taking $V_2\equiv 0$ one has  the following  results which recovers that in \cite{Mes} mentioned above, and improves 
and generalizes those in \cite{SP}:
 
 \begin{theorem}
\label{theorem10}

Let $w\in H^1(\R^n)$ be a solution of the equation \eqref{eigen}
with a complex potential $\widetilde V\in L^{\infty}(\R^{n})$ satisfying
\begin{equation}
\label{10a}
\widetilde V(x)=\widetilde V_1(x)+\widetilde V_2(x),
\end{equation}
such that
\begin{equation}
\label{10b}
|\widetilde V_1(x)|\leq \frac{c_1}{\langle x\rangle^{\alpha}}= \frac{c_1}{(1+|x|^2)^{\alpha/2}},\;\;\;\;\;\;\;0\leq \alpha<1/2,
\end{equation}
and $\widetilde V_2$ real valued and supported in $\{ x\in \R^n\,:\,|x|\geq 1\}$ such that
\begin{equation}
\label{10c}
-(\partial_r \widetilde V_2(x))^{-}\leq \frac{c_2}{|x|^{2\alpha}},\;\;\;\;\;\;\;\;\;\;a^{-}=min\{a;0\}.
\end{equation}
Then there exists a constant $\lambda_0=\lambda_0(\|\widetilde V\|_{L^{\infty}(\R^{n})};c_1;c_2;\alpha)>0$
such that if
\begin{equation}
\label{preserved-3}
\int_{\R^n}\,e^{\lambda_0\,|x|^p} \,|w(x)|^2\,dx < \infty,\;\;\;\;\;\;\text{with}\;\;\;\;\;\;\;p=(4-2\alpha)/3,
\end{equation}
then
\begin{equation}
\label{conclu3}
w\equiv 0.
\end{equation}

Moreover, if \eqref{10b} and \eqref{10c} holds $\alpha>1/2$ and \eqref{preserved-3} holds with $p=1$ and large $\lambda_0=\lambda_0(\|\widetilde V\|_{L^{\infty}(\R^{n})};c_1;\alpha)>0$, then $ w\equiv 0$.

\end{theorem}

In \cite{SP} under the hypotheses $\widetilde V_2=0$,  \eqref{10b} and  \eqref{preserved-3}, but for all $\lambda_0>0$,  on the complex potential $V(x,t)$  on Theorem \ref{theorem10} it was shown that the eigenfunction $w(x)$ solution of \eqref{eigen} corresponding to the real eigenvalue $\zeta$
satisfies  $ w\equiv 0$.

 We observe that the conclusion of Corollary \ref{corollary13} applies, i.e. if $\widetilde V_2=0$, $\alpha=1/2$ in \eqref{10b},
 and \eqref{preserved-3} holds for some $p>1$ and $\lambda_0>0$, then $w\equiv 0$. In this direction we have the following 
 improvement of the result in Theorem \ref{theorem10} concerning the case $\alpha=1/2$ in \eqref{10b} and \eqref{10c}.
 
 \begin{theorem}
\label{theorem20}

Let $w\in H^1(\R^n)$ be a solution of the equation \eqref{eigen}
with a  potential $\widetilde V\in L^{\infty}(\R^{n})$ satisfying
\begin{equation}
\label{20a}
\widetilde V(x)=\widetilde V_1(x)+\widetilde V_2(x),
\end{equation}
such that $\widetilde V_1$ is complex valued with 
\begin{equation}
\label{20b}
|\widetilde V_1(x)|\leq \frac{c_1}{\langle x\rangle^{1/2}}= \frac{c_1}{(1+|x|^2)^{1/4}},
\end{equation}
and $\widetilde V_2$ is real valued and supported in $\{ x\in \R^n\,:\,|x|\geq 1\}$ such that
\begin{equation}
\label{20c}
-(\partial_r \widetilde V_2(x))^{-}\leq \frac{c_2}{|x|},\;\;\;\;\;\;\;\;\;\;a^{-}=min\{a;0\}.
\end{equation}
Then there exists a constant $\lambda_0=\lambda_0(\|\widetilde V\|_{L^{\infty}(\R^{n})};c_1;c_2)>0$
such that if
\begin{equation}
\label{preserved-30}
\int_{\R^n}\,e^{\lambda_0\,|x|} \,|w(x)|^2\,dx < \infty,
\end{equation}
then
\begin{equation}
\label{conclu30}
w\equiv 0.
\end{equation}
\end{theorem}

 We observe that Theorem \ref{theorem20} is a stationary result (not a consequence of the time evolution results
 in Theorems \ref{theorem4} and \ref{theorem4b}) in which the ellipticity of the laplacian in \eqref{eigen} 
 plays an essential role. 
 
  The proof of Theorem \ref{theorem20} will be based in the following Carleman estimate :
 
 \begin{theorem}
\label{theorem20a} 

Let $\rho\in (0,1]$ and $\widetilde V$ as in Theorem \ref{theorem20}. Then there exists  $\tau_0=\tau_0(\rho;\|\widetilde V\|_{\infty};c_1;c_2)>0$ such that the inequality
\begin{equation}
\label{carleman}
\tau^{3/2}\,\|\,|x|^{-1/2} \,e^{\tau |x|} g \|_2\leq \| e^{\tau|x|} (\Delta g+\widetilde V g)\|_2
\end{equation}
holds for any $\tau\geq \tau_0$ and any $g\in C^{\infty}_0(\R^n - \overline{B_{\rho}(0)}\,)$.
\end{theorem}

 \vskip.1in
 
 We return  to the consequence of our time evolution results. Thus, combining  Theorem \ref{theorem2} and the comments before the statement of Theorem \ref{theorem10} one has that Theorem \ref{theorem10} 
 also applies to the solutions of the non-elliptic eigenvalue problem 
\begin{equation}
\label{eigen2}
\mathcal L_k w + \widetilde V(x) w=\zeta w,
\end{equation}
with $\mathcal L_k$ as in \eqref{opel} with complex potential $\widetilde V$ and $\zeta\in \R$.
 \vskip.1in

We shall employ the above results to study  
the possible profile of the concentration blow up phenomenon in solutions of the initial value problem (IVP) associated to the 
non-linear Schr\"odinger
equation
\begin{equation}
\label{E: NLS2}
\begin{aligned}
\begin{cases}
&i \partial_t u + \Delta u \pm |u|^{a} u=0,\;\;\;\;\;\;\;\;\;x\in\R^n,\;\;t\in \R, \;\;a>0,\\
&u(x,0)=u_0(x).
\end{cases}
\end{aligned}
\end{equation}
We observe that if $u(x,t)$ is a solution of \eqref{E: NLS2} then for all $\sigma>0$ 
\begin{equation}
\label{ab1}
u_{\sigma}(x,t)=\sigma^{2/a}\,u(\sigma x,\sigma^2t),
\end{equation}
is also a solution of \eqref{ab1} with data $u_{\sigma}(x,0)=\sigma^{2/a}\,u_0(x)$, so
\begin{equation}
\label{ab2}
\| D^{s} u_{\sigma}(x,0)\|_2=\sigma^{2/a-n/2+s}\| D^{s} u_0\|_2,
\end{equation}
where $D^s f(x)=(|\xi|^s\widehat f)^{\lor}(x),\;s\in\R.$ Thus, if $s_a/2-2/a$ the size of the data does not change by the scaling and one says that
\begin{equation}
\label{ab3}
\dot H^{n/2-2/a}(\R^n)= D^{n/2-2/a}L^2(\R^n),
\end{equation}
is a critical space for the IVP \eqref{E: NLS2}. The following result  
concerning  the local well-posedness of the IVP \eqref{E: NLS2} in the critical cases was established in \cite{CW2}.

\begin{TA}  Let $s_a/2-2/a,\;\;s_a\geq 0$ with 
 $[s_a]\leq a-1$ if $a$ is not an odd integer, then for each $u_0\in H^{s_a}(\R^n)$
there exist $T=T(u_0)>0$ and a unique solution $u=u(x,t) $ of the IVP \eqref{E: NLS2}
with
\begin{equation}
\label{1}
u\in C([-T,T]:H^{s_a}(\R^n))\cap L^q([-T,T]:L^p_{s_a}(\R^n)) =Z^{s_a}_T.
\end{equation}
Moreover, the map data $\to$ solution is locally continuous from $H^{s_a}(\R^n)$ into  $Z^{s_a}_T$.
\end{TA}

Above we have introduced the  notations :

 (a) for  $1<p<\infty$ and $s\in \R$
  \begin{equation}
 \label{ps}
 L^p_s(\R^n)\equiv (1-\Delta)^{-s/2} L^p(\R^n),\;\;\;\;\;\|\cdot\|_{s,p}\equiv \|(1-\Delta)^{s/2}\cdot\|_p,
 \end{equation}
 with $L^2_s(\R^n)=H^s(\R^n)$,

 (b) the indices $(q,p)$ in \eqref{1} are given by the Strichartz estimate \cite{Str2}, \cite{GiVe} :
 \begin{equation}
 \label{str}
( \int_{-\infty}^{\infty} \|e^{it\Delta}u_0\|_p^qdt)^{1/q}\leq c\|u_0\|_2,
 \end{equation}
 where
 $$
 \frac{n}{2}=\frac{2}{q}+\frac{n}{p},\;\;\;\;\;2\leq p\leq \infty,\;\;\text{if}\;\;\, n=1,\;\;2
 \leq p<2n/(n-2),\;\;\;\text{if}\;\;\,n\geq 2.
 $$

\vskip.1in

The pseudo-conformal transformation deduced in \cite{GiVe} shows  that if $u=u(x,t)$ is a solution of \eqref{E: NLS2}, then 
\begin{equation}
\label {otrasol}
v(x,t) = \frac{e^{i \omega |x|^2/4(\nu+\omega t)}}{(\nu+\omega t)^{n/2}}\,u\left(\frac{x}{\nu+\omega t},\frac{\gamma +\theta t}{\nu+\omega t}\right),  
\;\;\;\;\;\;\nu\theta-\omega\gamma=1,
\end{equation}
satisfies the equation
\begin{equation}
\label{E: NLS3}
i \partial_t v + \Delta v \pm \,(\nu+\omega t)^{an/2-2}|v|^{a} v=0.
\end{equation}

Hence, in the $L^2$-critical case $a=4/n$ the equations \eqref{E: NLS2} and \eqref{E: NLS3} are the same. Also in this case $a=4/n$ the pseudo-conformal transformation preserves both the space $L^2(\Rn)$ and the space $H^1(\Rn)\cap L^2(\Rn:|x|^2dx)$.
In particular, if we take $u(x,t) =e^{it}\,\varphi(x)$ the standing wave solution, i.e. $\varphi(x)$ being the unique
positive solution (ground state)  of the non-linear elliptic equation 
\begin{equation}
\label{elliptic}
-\varphi+\Delta \varphi + |\varphi|^{4/n}\varphi=0,\;\;\;\;x\in\Rn,
\end{equation}
it follows that
\begin{equation}
\label{bu}
v(x,t)= \frac{e^{it/(1-t)}\,e^{-i|x|^2/4(1-t)}}{(1-t)^{n/2}}\,\varphi\left(\frac{x}{1-t}\right),
\end{equation}
is a solution of \eqref{E: NLS2} with $a=4/n$  and $+$ sign in the nonlinear term (focussing case) which blows up at time $t=1$, i.e.
$$
\lim_{t\uparrow 1}\| \nabla\,v(\cdot,t)\|_2=\infty,
$$
and
$$
\lim_{t\uparrow  1} |v(\cdot,t)|^2=  c\, \delta(\cdot), \;\;\;\text{in the distribution sense}. 
$$

 Since it is known that positive solutions of the elliptic problem \eqref{elliptic} (in particular the  ground state)  have exponential decay (see \cite{St}, \cite{BeLi}), i.e.
$$
\varphi(x)\sim b_1 e^{-b_2|x|},\;\;\;\;\;\;\;\;\;\;b_1,\,b_2>0,
$$
the blow up solution $v(x,t)$ in \eqref{bu} satisfies 
\begin{equation}
\label{bound1}
|v(x,t)|\leq \frac{1}{(1-t)^{n/2}}\,Q \left( \frac{x}{1-t}\right),\;\;\;\;\;\;\;\;\;t\in (-1,1),
\end{equation}
with $\;Q(x)=b_1\,e^{-b_2|x|}$. One may ask if it is possible to have a faster  \lq\lq concentration profile"  in a solution of \eqref{E: NLS2}
with $a=4/n$ than the one  described in \eqref{bound1}. In other words,  whether or not \eqref{bound1} can hold  with
\begin {equation}
\label {bound2}
Q(x)=b_1\,e^{-b_2|x|^p},\;\;\;\;\;\;\;b_1,\,b_2>0,\;\;\;p>1,
\end{equation}
or
\begin {equation}
\label {bound3}
Q(x)=b_1\,e^{-b_3|x|},
\end{equation}
with $b_3$ sufficiently large. More generally for $a\geq 4/n$ one may ask if  a blow up solution 
$v(x,t)$ of \eqref{E: NLS2} can satisfy
\begin{equation}
\label{bound4}
|v(x,t)|\leq \frac{1}{(1-t)^{2/a}}\,Q\left( \frac{x}{1-t}\right),\;\;\;\;\;\;\;\;\;t\in (-1,1),
\end{equation}
with $\;Q(\cdot)$ as in \eqref{bound2} or as \eqref{bound3}.
Our next result shows that this is not the case.

\begin{theorem}
\label{theorem5}

Let $\,a\geq 4/n$. Let $v\in C((-1,1): H^{n/2-2/a}(\Rn))$ be a solution of the equation \eqref{E: NLS2}.
If  \eqref{bound4} holds with $\;Q(\cdot)$ as in \eqref{bound2} for some $p>1$ and $b_2>0$ or as \eqref{bound3}, then $v\equiv 0$.

\end{theorem}

 In \cite{ekpv10} we establish the result in Theorem \ref{theorem5} for $a=4/n$ and $p>4/3$.

Now we  consider the equation   in \eqref{E: NLS2} with the operator describing the dispersive relation $\mathcal L_k$ as in \eqref{e1b} being non-degenerate
but not  elliptic
\begin{equation}
\label{E: NLS2b}
i \partial_t u + \mathcal L_k u \pm |u|^{a} u=0,\;\;\;\;\;\;\;\;\;\;\;\;a>0.
\end{equation} 
In this case, the local well-posedness theory  is similar to that described above
for the IVP  \eqref{E: NLS2}. This follows from the fact that the local theory  is based on the Strichartz estimates in \eqref{str} which do not require the ellipticity of the laplacian, i.e. \eqref{str} holds with $\mathcal L_k$ instead of $\Delta$. Hence the results in \cite{CW2} still holds for the IVP associated
to the equation in 
 \eqref{E: NLS2b}. In addition, in this case the pseudo-conformal transformation tells us that if $u=u(x,t)$ is a solution of \eqref{E: NLS2b}, then 
\begin{equation}
\label {otrasolb}
v(x,t) = \frac{e^{i \omega \mathcal Q_k(x)/4(\nu+\omega t)}}{(\nu+\omega t)^{n/2}}\,u\left(\frac{x}{\nu+\omega t},\frac{\gamma +\theta t}{\nu+\omega t}\right),  
\;\;\;\;\;\;\nu\theta-\omega\gamma=1,
\end{equation}
with
\begin{equation}
\label{quadform}
\mathcal Q_k(x)=x_1^2+..+x_k^2-x_{k+1}^2-..-x_n^2,
\end{equation}
verifies the equation
\begin{equation}
\label{E: NLS3b}
i \partial_t v + \mathcal L_k v \pm \,(\nu+\omega t)^{an/2-2}|v|^{a} v=0.
\end{equation}

Hence, as in Theorem \ref{theorem5} we have:

\begin{theorem}
\label{theorem6}

Let $\,a\geq 4/n$. Let $v\in C((-1,1): H^{n/2-2/a}(\Rn))$ be a solution of the equation in \eqref{E: NLS2}.
If $u$ satisfies  \eqref{bound4} with $\;Q(\cdot)$ as in \eqref{bound2} or as \eqref{bound3}, then $u\equiv 0$.

\end{theorem}

It should be remarked that the result in Theorem \ref{theorem6}  is a conditional one. It assumes that the local solution 
of the IVP associated to the equation  \eqref{E: NLS2} blows up (see  \eqref{bound4}) which is a open problem.

We will adapt  our results in Theorems \ref{theorem4} and \ref{theorem4b} to study the possible profile of  \lq\lq generalized traveling wave" solutions 
of a class of equations containing  those in \eqref{E: NLS2} and \eqref{E: NLS2b}, (see \eqref{f1} and  \eqref{f2} below).  
Roughly, these are solutions $u(x,t)$ for which there exist  $\mu\in R$ and $\vec e\in \mathbb S^{n-1}$ such that the
$L^2(\R^n)$-norm of $u(x-\mu \,t \,\vec e, t)$ remains highly concentrated  at the origin for all time $t\geq 0$, see \eqref{11bc} and \eqref{b11bc}
 below.

\begin{corollary}
\label{corollary11}

Let $u\in C([0,\infty) :L^2(\Rn))$ be a solution of the equation \eqref{e1} or the equation \eqref{e1b}
with a real potential $V\in L^{\infty}(\R^{n}\times [0,\infty))$.

(a) If there exist  $\mu\in\R$ and $\vec e\in \mathbb S^{n-1}$ such that 
\begin{equation}
\label{potcon11a}
|V(x,t)|\leq  \frac{c_1}{(1+|x+ \mu\,t\,\vec e\, |^2)^{\alpha/2}},
\end{equation}
for some constants $c_1>0$ and  $\,\alpha\in [0,1/2)$. 
Then there exists $\lambda_0(\|V\|_{L^{\infty}(\R^{n}\times [0,\infty))}; \,c_1;\,\alpha)>0$ such that if 
\begin{equation}
\label{11bc}
\sup_{t\geq 0}\;\int_{\R^n}\,e^{\lambda_0\,|x+ \mu\,t\,\vec e\,|^{p}} \,|u(x,t)|^2\,dx < \infty,\;\;\;\;\;\;\text{with}\;\;\;\;\;\;\;p=(4-2\alpha)/3,
\end{equation}
then
\begin{equation}
\label{conclu11}
u\equiv 0.
\end{equation}

(b) If there exist  $\mu\in\R$ and $\vec e\in \mathbb S^{n-1}$ such that 
\begin{equation}
\label{bpotcon11a}
|V(x,t)|\leq  \frac{c_1}{(1+|x+ \mu\,t\,\vec e\,|^2)^{1/4+\epsilon_0/2}},\;\;\;\;\;\;\epsilon_0>0,
\end{equation}
for some constants $c_1>0$.
Then there exists $\lambda_0(\|V\|_{L^{\infty}(\R^{n}\times [0,\infty))}; \,c_1;\,\alpha)>0$ such that if 
\begin{equation}
\label{b11bc}
\sup_{t\geq 0}\;\int_{\R^n}\,e^{\lambda_0\,|x+ \mu\,t\,\vec e\,|}
 \,|u(x,t)|^2\,dx < \infty,
\end{equation}
then
\begin{equation}
\label{bconclu11}
u\equiv 0.
\end{equation}

\end{corollary}

 As in Corollary \ref{corollary13} we remark that if \eqref{potcon11a} holds with $\alpha=1/2$ and \eqref{11bc}
 holds for some $p>1$ and $\lambda_0>0$, then $ u\equiv 0$.  
  
 Finally we shall consider the semi-linear equations of the form 
 \begin{equation}
 \label{f1}
  \partial_t u=i(\Delta u+ F(u,\overline u) u),
  \end{equation}
 and
  \begin{equation}
 \label{f2}
  \partial_t u=i(\mathcal L_k u+ F(u,\overline u) u),
  \end{equation}
 with $\mathcal L_k$ as in \eqref{opel} and $F \,:\C^2\to\R$ (real valued), $F(0,0)=0$, and  such that there exists $M>0$ and $j\in \Z^+$ such that
  \begin{equation}
 \label{f3}
  |F(z,\overline z)|\leq M(|z|+|z|^j).
  \end{equation}

As a direct consequence of Theorems \ref{theorem4} and \ref{theorem4b}, Corollary \ref{corollary11}, and an appropriate version of the Galilean invariant property for solution of the equations \eqref{f1} and \eqref{f2} we shall establish the following result:
\begin{corollary}
\label{corollary12}

Let $u\in C([0,\infty):L^2(\Rn))$ be a solution of the equation \eqref{f1} or the equation \eqref{f2}.
 If there exist $\mu\in \R$ and $\vec e\in \mathbb S^{n-1}$ such that 
 \begin{equation}
 \label{ine}
 |u(x,t)|\leq Q(x+ \mu\,t\,\vec e),\;\;\;\;\;\;\;\forall \,x\in\R^n,\;\;t>0,
 \end{equation}
with $\;Q(\cdot)$ as in \eqref{bound2} for some $p>1$ or as in \eqref{bound3},  then $u\equiv 0$.

\end{corollary}

In \cite{GhSa} it was proved that the equation  \eqref{f2} with a $\mathcal L_k$ non-elliptic  operator  does not have  nontrivial (travelling wave) 
 solutions of the form
$$
u(x,t)=e^{i\omega t}\,\varphi(x+\mu t \vec e),\;\;\;\;\;\;\mu\in \R, \;\,\vec e\in \mathbb S^{n-1},
$$
with $\varphi \in H^1(\R^n)\cap H^2_{loc}(\R^n)$.

\vskip.05in
   The rest of this paper is organized as follows. Section 2 contains the details of the  proof 
of Theorem \ref{theorem4} in the case $V_2\equiv 0$ (the proof of Theorems \ref{theorem2}, \ref{theorem10}, \ref{theorem5}, and \ref{theorem6}, and Corollaries \ref{corollary11} and \ref{corollary12} follows this approach) and the  modifications needed in this proof to obtain the general case.
 The modifications of this argument required to establish
Theorems \ref{theorem4b} will be given in section 3. 
Also section 3 contains some remarks on the proof of Theorem \ref{theorem2}. 
Theorem \ref{theorem5} will be proved in section 4,
and the proofs of Corollaries \ref{corollary11}-\ref{corollary12} will be outlined in section 5. 
Finally, Theorems \ref{theorem20} and \ref{theorem20a} will be proven in section 6. The appendix is concerned with the existence of 
the functions $\varphi$ used in the proofs of Theorem \ref{theorem4} and Theorem \ref{theorem4b}.

\end{section}
\begin{section}{Proof of Theorem \ref{theorem4} \label{S2}}

We begin with two preliminary results.   Let $\mathcal S$ be a symmetric operator independent of $t$. Let  $\mathcal A$ be a skew-symmetric one.

\begin{proposition}
\label{proposition1}
For any $T_0 ,\, T_1\in\R,\;T_0<T_1$ and any suitable function $f(x,t)$ one has
\begin{equation}
\label{claim1}
\begin{aligned}
&\int_{T_0}^{T_1} \int  [\mathcal S;\mathcal A]f\,\overline f\,dxdt+
\int_{T_0}^{T_1} \int |\mathcal Sf|^2 dxdt\\
&\leq 
\int_{T_0}^{T_1} \int |\partial_t f -(\mathcal S+\mathcal A)f|^2 dx dt \\
&+
 |\int \mathcal S f(T_1)\overline{f(T_1)} dx |
 + |\int \mathcal S f(T_0)\overline{f(T_0)} dx |.
 \end{aligned}
 \end{equation}
\end{proposition}

\begin{proof}  

Since $\mathcal S$ is independent of $t$ one has
\begin{equation}
\label{AA}
\begin{aligned}
&\partial_t\langle \mathcal S f,f\rangle = \langle \partial_t f, \mathcal S f \rangle + \langle \mathcal Sf,\partial_t f\rangle\\
&=\langle \partial_t f-(\mathcal S+\mathcal A)f, \mathcal S f \rangle + \langle \mathcal Sf,\partial_t f-(\mathcal S+\mathcal A)f\rangle \\
&\;\;+\langle( \mathcal S+\mathcal A) f,\mathcal S f\rangle
+ \langle \mathcal S f,(\mathcal S +\mathcal A) f\rangle\\
&= 2\,\Re \langle \partial_t f-(\mathcal S+\mathcal A)f, \mathcal S f \rangle + 
2 \langle \mathcal S f,\mathcal S f\rangle 
+\langle [\mathcal S\mathcal A-\mathcal A\mathcal S]f,f\rangle.
\end{aligned}
\end{equation}
Thus, integrating in the time interval $[T_0,T_1]$ it follows that
$$
\aligned
&\int_{T_0}^{T_1} \langle [\mathcal S;\mathcal A]f,f\rangle dt + 2\int_{T_0}^{T_1} 
\langle \mathcal S f,\mathcal S f\rangle dt\\
&= - 2\; \Re \,\int_{T_0}^{T_1}  \langle \partial_t f-(\mathcal S+\mathcal A)f, \mathcal S f \rangle + 
\langle \mathcal Sf,f\rangle |_{T_0}^{T_1}. 
\endaligned
$$
Then, using that $2ab\leq a^2+b^2$ we obtain \eqref{claim1}.
\end{proof}

Next, for a fixed  $T\in\R$ we define $ \eta :[T-1/2,T+1/2]\to \R$ as
$$
\eta(t)=(t-(T-1/2))((T+1/2)-t),
$$
so $\eta(T-1/2)=\eta(T+1/2)=0$ and for any $t\in [T-1/2,T+1/2]$ 
$$
0\leq \eta(t)\leq 1/4,\;\;\;\;\;|\eta'(t)|\leq 1,\;\;\;\;\;\eta''(t)=-2.
$$

\begin{proposition}
\label{proposition2}
For any $T>1/2$ one has
\begin{equation}
\label{claim2}
\begin{aligned}
&\int_{T-1/2}^{T+1/2} \int  \eta(t) (|\mathcal Sf|^2+ [\mathcal S;\mathcal A]f\,\overline f )\,dxdt  + \int_{T-1/2}^{T+1/2} \int |f|^2dx dt\\\
&\leq 
8 \,\int_{T-1/2}^{T+1/2} \int |\partial_t f -(\mathcal S+\mathcal A)f|^2 dx dt 
 + 8\, |\int|f|^2dx|_{T-1/2}^{T+1/2}\,|.
 \end{aligned}
 \end{equation}
\end{proposition}

\begin{proof}  

Since 
\begin{equation}
\label{eq11}
\begin{aligned}
&\partial_t\langle  f,f\rangle = \langle \partial_t f, f \rangle + \langle f,\partial_t f\rangle\\
&=\langle \partial_t f-(\mathcal S+\mathcal A)f,  f \rangle + \langle f,\partial_t f-(\mathcal S+\mathcal A)f\rangle\\
& +
\langle( \mathcal S+\mathcal A) f, f\rangle
+ \langle  f,(\mathcal S +\mathcal A) f\rangle\\
&= 2\,\Re \langle \partial_t f-(\mathcal S+\mathcal A)f,  f \rangle + 
2 \langle \mathcal S f, f\rangle, 
\end{aligned}
\end{equation}
multiplying by $\eta'(t) $ and  integrating in the time interval $[T-1/2,T+1/2]$ one gets 
\begin{equation}
\label{eq11b}
\begin{aligned}
&-2 \,\int_{T-1/2}^{T+1/2} \eta'(t)\langle \mathcal S f,f\rangle dt\\
&= 2\Re  \,\int_{T-1/2}^{T+1/2}  \langle \partial_t f-(\mathcal S+\mathcal A)f, f \rangle \eta'(t) dt -  
\int_{T-1/2}^{T+1/2} \partial_t \langle f,f\rangle  \eta'(t) dt.
\end{aligned}
\end{equation}

Integration by parts gives
\begin{equation}
\label{eq12}
- \int_{T-1/2}^{T+1/2} \partial_t \langle f,f\rangle  \eta'(t) dt
= - \langle f,f \rangle \eta'(t)|_{T-1/2}^{T+1/2} + \int_{T-1/2}^{T+1/2} \langle f,f\rangle  \eta''(t) dt
\end{equation}
and
\begin{equation}
\label{eq13}
-2 \,\int_{T-1/2}^{T+1/2} \eta'(t)\langle \mathcal S f,f\rangle dt= 2 
\,\int_{T-1/2}^{T+1/2} \eta(t)\,\partial_t\langle \mathcal S f,f\rangle dt.
\end{equation}

We recall that from  \eqref{AA} one has 
\begin{equation}
\label{eq14}
\partial_t \langle \mathcal Sf,f\rangle =
2\,\Re \langle \partial_t f-(\mathcal S+\mathcal A)f, \mathcal S f \rangle + 
2 \langle \mathcal S f,\mathcal S f\rangle 
+
\langle [\mathcal S;\mathcal A]f,f\rangle,
\end{equation}
so inserting \eqref{eq14} into \eqref{eq13}, and the result together with  \eqref{eq12}
into \eqref{eq11b} it follows that
\begin{equation}
\label{eq15}
\begin{aligned}
&4 \,\int_{T-1/2}^{T+1/2} \eta(t)  \langle \mathcal S f,\mathcal S f\rangle dt 
+ 2 \,\int_{T-1/2}^{T+1/2} \eta(t) \langle [\mathcal S;\mathcal A]f,f\rangle dt\\
&=- 4 \Re \int_{T-1/2}^{T+1/2} \eta(t) \langle \partial_t f-(\mathcal S+\mathcal A)f, \mathcal S f \rangle dt\\
&+ 2\,\Re  \,\int_{T-1/2}^{T+1/2}  \langle \partial_t f-(\mathcal S+\mathcal A)f, f \rangle \eta'(t) dt \\
&- \langle f,f \rangle \eta'(t)|_{T_1/2}^{T+1/2} + \int_{T-1/2}^{T+1/2} \langle f,f\rangle  \eta''(t) dt,
\end{aligned}
\end{equation}
which combined with the properties of the function 
$\eta$ and Cauchy-Schwarz yields the estimates \eqref{claim2}.

\end{proof}

\underline{Proof of Theorem \ref{theorem4}: case $V_2\equiv 0$.}

\vskip.1in
  We  fix $\alpha\in[0,1/2)$ and $p=(4-2\alpha)/3\in(1,4/3]$. Let $\varphi=\varphi_p$ be a $C^4,$ radial, strictly convex 
  function on compact sets of $\R^n$,
such that 
\begin{equation}
\label{772}
\begin{aligned}
& \varphi(r)=r^p+\beta, \;\;\;\;\text{for}\;\;\;r=|x|\geq 1,\\
&\varphi(0)=0,\;\;\;\;\;\;\;\;\;\varphi(r)>0,\;\;\;\text{for}\;\;\;r>0,\\
&\;\exists \,M>0 \;\;\;\;\text{s.t.} \;\;\;\varphi(r)\leq M r^p,\;\;\;\,\forall \;r\in[0,\infty).
\end{aligned}
\end{equation}
The existence of such a function $\varphi=\varphi_p$ will be discussed in the Appendix, part (a).

We recall that
\begin{equation}
\label{hessiana}
D^2\varphi = \partial^2_r\varphi \left(\frac{x_jx_k}{r^2}\right)+
\frac{\partial_r\varphi}{r}\left(\delta_{jk}-\frac{x_jx_k}{r^2}\right).
\end{equation}
Therefore,
\begin{equation}
\label{hessian2a}
\nabla \varphi D^2\varphi \nabla \varphi =\partial^2_r\varphi (\partial_r\varphi)^2
=\frac{c}{\; |x|^{4-3p}},\;\;\;\;\;\;\;\text{for}\;\;\;\;\;\;r=|x|\geq 1,
\end{equation}
and
\begin{equation}
\label{hessian2'a}
 D^2\varphi \geq p(p-1)r^{p-2} I,\;\;\;\;\;\;\;\text{for}\;\;\;\;\;\;r=|x|\geq 1.
\end{equation}

Let  $f(x,t)=e^{\lambda \varphi(x)}u(x,t)$ where $u(x,t)$ is a solution of the IVP \eqref{e1}  so
\begin{equation}
\label{eq1}
e^{\lambda \varphi}(\partial_t-i \Delta) u = e^{\lambda \varphi}(\partial_t- i \Delta) (e^{-\lambda \varphi}f)=\partial_t f -\mathcal S f -\mathcal A f,
\end{equation}
where $\mathcal S$ is symmetric and $\mathcal A$ skew-symmetric both independent of $t$ with
\begin{equation}
\label{defsa}
\mathcal S=-i\lambda (2\nabla \varphi\cdot \nabla +\Delta \varphi),\;\;\;\;\;\;\;\mathcal A=i(\Delta +\lambda^2|\nabla\varphi|^2),
\end{equation}
so that
\begin{equation}
\label{comm}
[\mathcal S;\mathcal A]=-\lambda ((4\nabla\cdot\,D^2\varphi\nabla\;)-4\lambda^2 \nabla \varphi D^2\varphi\nabla\varphi+\Delta^2\varphi).
\end{equation}

We divide the proof into three steps:

\underline {Step 1} : If 
\begin{equation}
\label{clambda}
\sup_{t>0}\,\int \,e^{\lambda |x|^p}|u(x,t)|^2 dx\leq c_{\lambda},\;\;\;\;\;\;\;\;p=(4-2\alpha)/3.
\end{equation}
Then  there exists $\{T_j\,:\,j\in \Z^+\}$ with $T_j\uparrow \infty$ as $j\uparrow \infty$  such that
\begin{equation}
\label{clambda2}
\sup_{j\in\Z^+}\,\int \,|\mathcal S f(x,T_j)|^2 dx\leq \widetilde c_{\lambda},
\end{equation}
where
$$
f=e^{\lambda \varphi(x)}u(x,t),
$$
$\mathcal S$ as in \eqref{defsa}, and $\widetilde c_{\lambda}$ denoting a constant depending on $c_{\lambda}$ in \eqref{clambda},
$\lambda$, $\|V\|_{\infty}$ and $p$.

\underline{Proof of step 1} : We combine Proposition \ref{proposition2} 
with \eqref{comm} passing the term involving $\Delta^2\varphi$ to the right hand side and using that the rest of 
the commutator in \eqref{comm} is positive to obtain
\begin{equation}
\label{step1-eq1}
\begin{aligned}
\int_{T-1/2}^{T+1/2}\int &|\mathcal Sf|^2\eta(t)dx dt\leq 
8\,(\int_{T-1/2}^{T+1/2}\int|\partial_tf-\mathcal Sf-\mathcal Af|^2dxdt\\
&+\lambda \|\Delta^2\varphi\|_{\infty} \int_{T-1/2}^{T+1/2}\int | f|^2 dx dt  + |\int|f|^2dx|^{T+1/2}_{T-1/2}\,|)\equiv B.
\end{aligned}
\end{equation}

We use that
$$
e^{\lambda \varphi}(\partial_t-i \Delta)u=\partial_t f -\mathcal S f -\mathcal A f,\;\;\;\;\;\;(\partial_t - i \Delta) u = i V u,
$$
to bound the right hand side of \eqref{step1-eq1} as
\begin{equation}
\label{step1-eq2}
B\leq c( \lambda \| \Delta^2\varphi\|_{\infty} +\sup_{t>0}\|V(\cdot,t)\|^2_{\infty})\, \sup_{t>0}\,\int \,e^{2\lambda \varphi}|u(x,t)|^2dx
\leq \tilde c_{\lambda}.
\end{equation}
Inserting this in \eqref{step1-eq1} and using that $\eta(t)\geq 3/16$ for $t\in [T-1/4,T+1/4]$ one gets that
$$
\aligned
&\tilde c_{\lambda}\geq \int_{T-1/2}^{T+1/2}\int |\mathcal Sf|^2\eta \,dx dt\geq \int_{T-1/4}^{T+1/4}\int |\mathcal Sf|^2\eta \,dx dt
\\
&\geq \frac{3}{16} \int_{T-1/4}^{T+1/4}\int |\mathcal Sf|^2dx  dt\geq \frac{3}{32}\int |\mathcal Sf(x,T^*)|^2dx ,
\endaligned
$$
for some $T^*\in [T-1/4,T+1/4]$. Hence, we can find a sequence $\{T_j\,:\,j\in \Z^+\}$ with $t_j\uparrow \infty$ an $j\uparrow \infty$ such that
\begin{equation}
\label{step1final}
\sup_{j\in\Z^+} \,\int |\mathcal Sf(x,T_j)|^2dx \leq \tilde c_{\lambda}.
\end{equation}

\underline {Step 2} : There exists $\lambda _0>0$ such that if $\lambda \geq \lambda_0$, then for any $j\in\Z^+$,
\begin{equation}
\label{step2-eq1}
\int_{T_1}^{T_j}\int \frac{ e^{2\lambda \varphi(x)}\,|u(x,t)|^2}{\langle x\rangle^{4-3p}}\,dx dt \leq \tilde c_{\lambda}\;\;\;\;\;\;\text{uniformly in }\;j\in\Z^+.
\end{equation}

\underline{Proof of step 2} :  A combination of  Proposition \ref{proposition1}, the conclusion of step 1, and  our hypothesis  
leads to
\begin{equation}
\label{step2-eq1a}
\int_{T_1}^{T_j}\int [\mathcal S;\mathcal A] f \overline{f} dx dt \leq 
\int_{T_1}^{T_j}\int |e^{\lambda \varphi} V u|^2 dx dt + \tilde c_{\lambda}.
\end{equation}
From our hypothesis \eqref{772} on $\varphi$ one has that
\begin{equation}
\label{step2-eq2}
\begin{aligned}
&\nabla \varphi D^2\varphi \nabla \varphi\geq \frac{c}{\; |x|^{4-3p}},\;\;\;\;\;|x|\geq 1,\\
&|\Delta^2 \varphi(x)|\leq \frac{c}{ \langle x\rangle^2},\;\;\;\;\;\;\;\;\;\;\;\forall x\in \R^n.
\end{aligned}
\end{equation}
Thus, from our decay hypothesis on the potential \eqref{potcon2bc} it follows that there exists $\widetilde \lambda>0$ such 
that if $\lambda \geq\widetilde \lambda$ and $|x|\geq 1$, then
\begin{equation}
\label{007}
2\lambda^2 \nabla \varphi D^2\varphi \nabla \varphi+\Delta^2 \varphi -|V|^2\geq \frac{\lambda}{\langle x\rangle^{4-3p}} .
\end{equation}

Next, for any $\epsilon\in(0,1)$ we consider the domain $\{x\,:\,\epsilon\leq|x|\leq1\}$. In this set we have that
\begin{equation}
\label{betweene1}
\nabla \varphi D^2\varphi \nabla \varphi \geq c_{\varphi,\epsilon},\;\;\;\;\;\text{for}\;\;\;\;\epsilon\leq|x|\leq1.
\end{equation}
Therefore, for large enough $\lambda\geq \lambda_{\epsilon}$ 
\begin{equation}
\label{00711}
\lambda^2 \nabla \varphi D^2\varphi \nabla \varphi+\Delta^2 \varphi -|V|^2\geq \lambda,\;\;\;\;\;\text{for}\;\;\;\;\epsilon\leq|x|\leq1.
\end{equation}

Hence from \eqref{step2-eq1a} 
\begin{equation}
\label{step2-eq3}
\begin{aligned}
&4 \lambda\int_{T_1}^{T_j}\int \nabla f D^2\varphi \nabla\overline{f} dxdt +
2\lambda ^3 \int_{T_1}^{T_j}\int \nabla \varphi D^2\varphi \nabla \varphi |f|^2 dxdt\\
&\leq \tilde c_{\lambda}+ c' \,(\lambda \|\Delta^2\varphi\|_{\infty}+\|V\|^2_{\infty}) \int_{T_1}^{T_j}\int_{|x|\leq \epsilon} |f|^2 dxdt.
\end{aligned}
\end{equation}

In the domain $\{x\,:\,|x|\leq\epsilon\}$ we shall use that $\varphi $ is strictly convex in $r=|x|\leq 2$ to get from \eqref{step2-eq3} that 
\begin{equation}
\label{step2-eq4}
\begin{aligned}
&4 c_{\varphi}\lambda\int_{T_1}^{T_j}\int_{|x|\leq 2\epsilon}  |\nabla f|^2 dxdt +
2\lambda ^3 \int_{T_1}^{T_j}\int \nabla \varphi D^2\varphi \nabla \varphi |f|^2 dxdt\\
&\leq \tilde c_{\lambda} + c'\, (\lambda \|\Delta^2\varphi\|_{\infty}+\|V\|^2_{\infty})\,\int_{T_1}^{T_j}\int_{|x|\leq \epsilon} |f|^2 dxdt,
\end{aligned}
\end{equation}
with $c_{\varphi}$ and $c'$ independent of $\epsilon\in(0,1]$.
Now, we pick $\theta \in C^{\infty}(\R^n)$ such that $\theta(x)\equiv 1$ for $|x|\leq \epsilon$ with $supp\,\theta\subset \{x\,:\,|x|\leq 2 \epsilon\}$ and use Poincare's inequality to get that for each $t\in [T_1,T_j]$ 
\begin{equation}
\label{step2-eq4b} 
\begin{aligned}
&\int_{|x|\leq \epsilon} |f|^2dx\leq \int_{|x|\leq 2\epsilon}|\theta f|^2dx\leq c_{\varphi}\, \epsilon^2 \int_{|x|\leq 2\epsilon} |\nabla (\theta f)|^2 dx \\
&\leq  c_{\varphi}\, \epsilon^2 \int _{|x|\leq 2\epsilon} |\nabla  f |^2 dx + c_{\varphi} \int_{\epsilon\leq |x|\leq 2\epsilon} |f|^2 dx.
\end{aligned}
\end{equation}

 Fixing $\epsilon$ sufficiently small and then $\lambda$ large enough it follows from this that 
 \begin{equation}
\label{AAAA} 
\begin{aligned}
& \lambda\int_{T_1}^{T_j}\int \nabla f D^2\varphi \nabla \overline{f} dxdt + \lambda^2 \int_{T_1}^{T_j}\int_{|x|\leq 1} |f|^2 dxdt \\
&+\lambda ^3 \int_{T_1}^{T_j}\int \nabla \varphi D^2\varphi \nabla \varphi |f|^2 dxdt  \leq \tilde c_{\lambda}.
\end{aligned}
\end{equation}
In particular,  for $\lambda_0\geq \widetilde \lambda$ sufficiently large we have
\begin{equation}
\label{step2-eq5} 
 \int_{T_1}^{T_j} \,\int\,\frac{|f|^2}{\langle x\rangle^{4-3p}}dxdt\leq \widetilde c_{\lambda},\;\;\;\text{uniformly in}\;j\in\Z+,
\end{equation}
which completes the proof of this step.

\vskip.1in

We fix $\;\lambda=\lambda_0$ above for the rest of the proof.
\vskip.1in

\underline {Step 3} : $u(x,t)\equiv 0$. 

\underline{Proof of step 3} : On the one hand, since the potential $V=V(x,t)$ is real, then the $L^2$-norm of the solution $u(x,t)$
of \eqref{e1} is preserved, i.e. for all $t\in \R$ 
$$
\|u(\cdot,t)\|_2=\|u_0\|_2.
$$
On the other hand, from step 2 inequality \eqref{step2-eq1} one has 
$$
\aligned
(T_j-T_1)\|u_0\|_2^2&=\int_{T_1}^{T_j}\int |u(x,t)|^2dx dt\\
&=\int_{T_1}^{T_j} \int  |u(x,t)|^2\frac{e^{2\lambda \varphi}}{\langle x\rangle^{4-3p}}\,
\langle x\rangle^{4-3p} e^{-2\lambda \varphi}\,dx dt\\
&\leq \sup_{x\in\R^n}(\langle x\rangle^{4-3p} e^{-2\lambda \varphi})\,
\int_{T_1}^{T_j}\int  |u(x,t)|^2\frac{e^{2\lambda \varphi}}{\langle x\rangle^{4-3p}} dx dt\leq \widetilde c_{\lambda_0},
\endaligned
$$
which completes the proof of Theorem \ref{theorem4} in the case $V_2\equiv 0$.
\vskip.1in

\underline{Proof of Theorem \ref{theorem4}: general case.}
\vskip.1in
The argument is similar to that presented above in the case $V_2\equiv 0$, so we sketch it. The step 1 is similar so it will be omitted. In the step 2 we divide the potential $V(x,t)$ as in  \eqref{split-pot},
$$
V(x,t)=V_1(x,t)+V_2(x,t),
$$ and define
\begin{equation}
\label{newAS}
\mathcal S=-i\lambda (2\nabla \varphi\cdot \nabla +\Delta \varphi),\;\;\;\;\;\;\;\mathcal A=i(\Delta +V_2+\lambda^2|\nabla\varphi|^2),
\end{equation}
so that
\begin{equation}
\label{comm2}
[\mathcal S;\mathcal A]=-\lambda ((4\nabla\cdot\,D^2\varphi\nabla\;)-4\lambda^2 \nabla \varphi D^2\varphi\nabla\varphi+\Delta^2\varphi)+ 2\lambda \nabla\varphi\cdot \nabla V_2=D_1+D_2.
\end{equation}

We notice that $D_1$ is similar to the term handled in the proof of Theorem \ref{theorem4} in the case $V_2\equiv 0$, and that since $\varphi$ is radial and 
convex one has
\begin{equation}
\label{con12}
D_2=2\lambda \nabla\varphi \cdot\nabla V_2=2\lambda \partial_r\varphi\partial_r V_2\geq 2\lambda  \partial_r\varphi
(\partial_r V_2)^{-}.
\end{equation}
\vskip.1in
Thus, from our decay hypothesis on the potential it follows that there exists $\lambda_0>0$ such 
that if $\lambda \geq \lambda_0$ and $|x|\geq 1$, then
\begin{equation}
\label{007b}
2\lambda^2 \nabla \varphi D^2\varphi \nabla \varphi+\Delta^2 \varphi -|V_1|^2
+2\partial_r\varphi
(\partial_r V_2)^{-}\geq \frac{\lambda}{\langle x\rangle^{4-3p}} .
\end{equation}
For $|x|\leq 1$ we apply the argument in the proof of Theorem \ref{theorem4} in the case $V_2\equiv 0$. Therefore combining these
estimates we obtain the proof of the step 2 :
There exists $\lambda _0>0$ such that if $\lambda \geq \lambda_0$, then for any $j\in\Z^+$
\begin{equation}
\label{step2-eq1b}
\int_{T_1}^{T_j}\int \frac{ e^{2\lambda \varphi(x)}\,|u(x,t)|^2}{\langle x\rangle^{4-3p}}\,dx dt \leq \tilde c_{\lambda}\;\;\;\;\;\;\text{independent of }\;j\in\Z^+.
\end{equation}

Once \eqref{step2-eq1b} has been established the rest of the proof follows the same argument given in the step 3
of the proof of Theorem \ref{theorem4} in the case $V_2\equiv 0$.
\end{section}
\begin{section}{Proofs of Theorem \ref{theorem4b} and Theorem \ref{theorem2}   \label{S3}}
\underline{Proof of Theorem \ref{theorem4b}: case $V_2\equiv 0.$}
\vskip.1in

We shall follow the argument provided in the proof Theorem \ref{theorem4}. A main difference  is the choice of the function 
$\varphi$ in \eqref{772}. In this case we take $\varphi\in C^4\,$ to be  a radial, strictly convex function on compact sets of $\R^n$, such that 
\begin{equation}
\label{newvarphi}
\varphi(r)=3 r-\int_1^r\,\frac{dr}{1+\log r}+\beta,\;\;\;\;\;\;\;\;\;\;r=|x|\geq 1,
\end{equation}
so
\begin{equation}
\label{newvarphi1}
\partial_r\varphi(x)=3-\frac{1}{1+\log r},\;\;\;\;\partial^2_r\varphi(x)=\frac{1}{r(1+\log r)^2},
\;\;\;\;r=|x|\geq 1,\\
\end{equation}
and
\begin{equation}
\label{772bb}
\begin{aligned}
& \varphi(0)=0,\;\;\;\;\;\varphi(r)>0,\;\;\;\text{for}\;\;r>0,\\
&\;\exists\, M>0 \;\;\;\;\text{s.t.} \;\;\;\varphi(r)\leq M r,\;\;\;\,\forall \;r\in[0,\infty).
\end{aligned}
\end{equation}
The existence of such a function $\varphi$ will be proven in the Appendix, part (b). 
Since
\begin{equation}
\label{hessian}
D^2\varphi = \partial^2_r\varphi \left(\frac{x_jx_k}{r^2}\right)+
\frac{\partial_r\varphi}{r}\left(\delta_{jk}-\frac{x_jx_k}{r^2}\right),
\end{equation}
for $|x|\geq 1$ one has 
\begin{equation}
\label{hessian2}
\nabla \varphi D^2\varphi \nabla \varphi =\partial^2_r\varphi (\partial_r\varphi)^2
>\frac{1}{r\,(1+\log r)^2 },
\end{equation}
and
\begin{equation}
\label{hessian2aaa}
D^2\varphi \geq \partial^2_r\varphi(x)I.
\end{equation}

The step 1 is similar to that in the proof of Theorem \ref{theorem4}, with the appropriate modifications, hence we shall start with  step 2. \newline
\underline {Step 2} : There exists $\lambda _0>0$ such that if $\lambda \geq \lambda_0$, then for any $j\in\Z^+$
\begin{equation}
\label{step2b-eq1}
\int_{T_1}^{T_j}\int \frac{ e^{2\lambda \varphi(x)}\,|u(x,t)|^2}{\langle x\rangle\,(\log\langle x\rangle)^2}\,dx dt \leq \tilde c_{\lambda}
\;\;\;\;\;\;\text{independent of }\;j\in\Z^+.
\end{equation}

\underline{Proof of step 2} :  A combination of  Proposition \ref{proposition1}, the conclusion of step 1, and  our hypothesis  
leads to
\begin{equation}
\label{step2b-eq1ab}
\int_{T_1}^{T_j}\int [\mathcal S;\mathcal A] f \overline{f} dx dt \leq 
\int_{T_1}^{T_j}\int |e^{\lambda \varphi} V u|^2 dx dt + \tilde c_{\lambda}.
\end{equation}
From our assumptions  on $\varphi$ it follows that
\begin{equation}
\label{step2b-eq2}
|\Delta^2 \varphi(x)|\leq \frac{c}{ \langle x\rangle^2},\;\;\;\;\;\;\;\;\;\;\;\forall x\,\in \R^n.
\end{equation}
Using the decay hypothesis on the potential \eqref{potcon2bcd} one has  that there exists $\widetilde \lambda>0$ such 
that if $\lambda \geq\widetilde \lambda$ and $|x|\geq 1$, then
\begin{equation}
\label{007bb}
2\lambda^2 \nabla \varphi D^2\varphi \nabla \varphi+\Delta^2 \varphi -|V|^2\geq \frac{\lambda}{\,r\,(1+\log r)^2 } .
\end{equation}

Thus,  from \eqref{step2b-eq1ab} and $\lambda>>1$
\begin{equation}
\label{step2b-eq3}
\begin{aligned}
&4 \lambda\int_{T_1}^{T_j}\int \nabla f D^2\varphi \nabla\overline{f} dxdt +
2 \lambda ^3 \int_{T_1}^{T_j}\int \nabla \varphi D^2\varphi \nabla \varphi |f|^2 dxdt\\
&\leq \tilde c_{\lambda}+ c\,(\lambda \|\Delta^2\varphi\|_{\infty}+\|V\|_{\infty}) \int_{T_1}^{T_j}\int_{|x|\leq 1} |f|^2 dxdt\\
&\leq \tilde c_{\lambda}+ c\,\lambda\, \int_{T_1}^{T_j}\int_{|x|\leq 1} |f|^2 dxdt. 
\end{aligned}
\end{equation}

Next, for a fixed $\epsilon\in(0,1)$ we consider the domain $\{x\,:\,\epsilon\leq|x|\leq1\}$. In this region
\begin{equation}
\label{betweene1aa}
\nabla \varphi D^2\varphi \nabla \varphi \geq c_{\varphi,\epsilon},\;\;\;\;\;\text{for}\;\;\;\;\epsilon\leq|x|\leq1.
\end{equation}
Therefore, for large enough $\lambda\geq \lambda_{\epsilon}$ 
\begin{equation}
\label{00711a}
\lambda^2 \nabla \varphi D^2\varphi \nabla \varphi+\Delta^2 \varphi -|V|^2\geq \lambda,\;\;\;\;\;\text{for}\;\;\;\;\epsilon\leq|x|\leq1.
\end{equation}

Hence  
\begin{equation}
\label{step2-eq3aa}
\begin{aligned}
&4 \lambda\int_{T_1}^{T_j}\int \nabla f D^2\varphi \nabla \overline{f} dxdt +
2\lambda ^3 \int_{T_1}^{T_j}\int \nabla \varphi D^2\varphi \nabla \varphi |f|^2 dxdt\\
&\leq \tilde c_{\lambda}+ c' \,\lambda \,\int_{T_1}^{T_j}\int_{|x|\leq \epsilon} |f|^2 dxdt,
\end{aligned}
\end{equation}
with $c'$ independent of $\epsilon\in(0,1]$.
In the domain $\{x\,:\,|x|\leq\epsilon\}$ we shall use that $\varphi $ is strictly convex in $r=|x|\leq 2$ to get from \eqref{step2b-eq3} that 
\begin{equation}
\label{step2-eq4ba}
\begin{aligned}
&2 c_{\varphi}\lambda\int_{T_1}^{T_j}\int_{|x|\leq 2\epsilon}  |\nabla f|^2 dxdt +
\lambda ^3 \int_{T_1}^{T_j}\int \nabla \varphi D^2\varphi \nabla \varphi |f|^2 dxdt\\
&\leq \tilde c_{\lambda} + c'\, \lambda \,\int_{T_1}^{T_j}\int_{|x|\leq \epsilon} |f|^2 dxdt,
\end{aligned}
\end{equation}
with $c_{\varphi}$ and $c'$ independent of $\epsilon\in(0,1]$.
Choosing $\theta \in C^{\infty}(\R^n)$ such that $\theta(x)\equiv 1$ for $|x|\leq \epsilon$ with $supp\,\theta\subset \{x\,:\,|x|\leq 2 \epsilon\}$ and using Poincare's inequality to get that for each $t\in [T_1,T_j]$ it follows that
\begin{equation}
\label{step2-eq4baa} 
\begin{aligned}
&\int_{|x|\leq \epsilon} |f|^2dx\leq \int_{|x|\leq 2\epsilon}|\theta f|^2dx\leq c_{\varphi}\, \epsilon^2 \int_{|x|\leq 2\epsilon} |\nabla (\theta f)|^2 dx \\
&\leq  c_{\varphi}\, \epsilon^2 \int _{|x|\leq 2\epsilon} |\nabla  f |^2 dx + c_{\varphi} \int_{\epsilon\leq |x|\leq 2\epsilon} |f|^2 dx.
\end{aligned}
\end{equation}

Gathering  the above estimates by fixing $\epsilon$ sufficiently small and then $\lambda>\widetilde \lambda$ large enough one concludes that
 \begin{equation}
\label{AAAAA} 
\begin{aligned}
& \lambda\int_{T_1}^{T_j}\int \nabla f D^2\varphi \nabla f dxdt + \lambda^2 \int_{T_1}^{T_j}\int_{|x|\leq 1} |f|^2 dxdt \\
&+\lambda ^3 \int_{T_1}^{T_j}\int \nabla \varphi D^2\varphi \nabla \varphi |f|^2 dxdt  \leq \tilde c_{\lambda}.
\end{aligned}
\end{equation}
In particular
\begin{equation}
\label{step2b-eq5ba} 
 \int_{T_1}^{T_j} \,\int\,\frac{|f|^2}{\langle x\rangle\,(\log\langle x\rangle)^2}dxdt\leq \widetilde c_{\lambda},\;\;\;\text{independent of }\;j\in\Z+,
\end{equation}
which completes the proof of the  step 2.

\vskip.1in

We fixed $\;\lambda=\lambda_0$ above for the rest of the proof.
\vskip.1in

\underline {Step 3} : $u(x,t)\equiv 0$ 

\underline{Proof of step 3} : On one hand, since the potential $V=V(x,t)$ is real, then the $L^2$-norm of the solution $u(x,t)$
of \eqref{e1} is preserved, i.e. for all $t\in \R$ 
$$
\|u(\cdot,t)\|_2=\|u_0\|_2.
$$
On the other hand, from step 2 \eqref{step2b-eq1}
$$
\aligned
(T_j-T_1)\|u_0\|_2^2&=\int_{T_1}^{T_j}\int |u(x,t)|^2dx dt\\
&=\int_{T_1}^{T_j} \int  |u(x,t)|^2\frac{e^{2\lambda \varphi}}{\langle x\rangle\,(\log\langle x\rangle)^2}\,
{\langle x\rangle\,(\log\langle x\rangle)^2}\,e^{-2\lambda \varphi}\,dx dt\\
&\leq \sup_{x\in\R^n}({\langle x\rangle\,(\log\langle x\rangle)^2}\,e^{-2\lambda \varphi})\,
\int_{T_1}^{T_j}\int  |u(x,t)|^2\frac{e^{2\lambda \varphi}}{\langle x\rangle\,(\log\langle x\rangle)^2}\,dx dt\\
&\leq \widetilde c_{\lambda_0},
\endaligned
$$
which completes the proof of Theorem \ref{theorem4b} in the case $V_2\equiv 0$.

The proof in the general case follows the same argument already explained in the proof of Theorem \ref{theorem4} so it will be omitted.

\vskip.1in
\underline{Proof of Theorem \ref{theorem2}}
\vskip.1in

 The only differences with the previous cases  are following computations:
 $$
 \mathcal S=-i\,\lambda(2\nabla\varphi\cdot\widetilde \nabla+\mathcal L_k),\;\;\;\;\;\;\;\widetilde \nabla=(\partial_{x_1},..,\partial_{x_k},-\partial_{x_{k+1}},..,-\partial_{x_n}),
 $$
 $$
 \mathcal A=i(\mathcal L_k+\lambda^2((\partial_{x_1}\varphi)^2+..+(\partial_{x_k}\varphi)^2-(\partial_{x_{k+1}}\varphi)^2-(\partial_{x_n}\varphi)^2),
 $$
so
$$
[\mathcal S;\mathcal A]= -\lambda((4\widetilde\nabla\cdot D^2\varphi\widetilde\nabla )-4\lambda^2\widetilde\nabla\varphi D^2\varphi\widetilde\nabla
\varphi+\mathcal L_k\mathcal L_k\varphi. 
$$
Hence, the method of proof used in Theorems \ref{theorem4}-\ref{theorem4b} for the elliptic case $\mathcal L_k=\Delta$ can be applied to obtain the same results in this non-degenerate case.
\end{section}
\begin{section}{Proof of Theorem \ref{theorem5} \label{S4}}

 The conformal transformation \eqref{otrasol} with $\nu=\omega=\theta=1$ and $\gamma=0$ 
tells us that
\begin{equation}
\label{newsolution}
w(x,t)=\frac{1}{(1+t)^{n/2}} \,e^{i|x|^2/4(1+t)}\,v(\frac{x}{1+t},\frac{t}{1+t}),
\end{equation}
solves the equation 
\begin{equation}
\label{000}
i \partial_t w + \Delta w \pm \,(1+ t)^{an/2-2}|w|^{a} w=0,
\end{equation}
in the time interval $t\in [0,\infty)$.  Thus, from the hypotheses \eqref{bound1}
it follows that the solution $w(x,t)$ satisfies
\begin{equation}
\label{main}
\begin{aligned}
&|w(x,t)| =\frac{1}{(1+t)^{n/2}}\,\left|v(\frac{x}{1+t},\frac{t}{1+t})\right|\\
&\leq \frac{1}{(1+t)^{n/2}} \,\frac{1}{(1-\tfrac{t}{(1+t)})^{2/a}} Q\left(\frac{\tfrac{x}{(1+t)}}{1-\tfrac{t}{(1+t)}}\right) 
= \frac{1}{(1+t)^{n/2-2/a}}\,Q(x).
\end{aligned}
\end{equation}
Since the potential $V(x,t)$ has the form
$$
V(x,t)=\pm (1+t)^{an/2-2}|w(x,t)|^{a},
$$
from \eqref{main} one sees that it verifies that
\begin{equation}
\label{potential-appl.}
|V(x,t)|\leq (1+t)^{an/2-2}   \left(\frac{1}{(1+t)^{n/2-2/a}}\right)^a\,Q^a(x)= Q^a(x).
\end{equation}

Therefore, since $a\geq 4/n>0$ from our hypothesis \eqref{bound2} or \eqref{bound3} it follows that the potential in \eqref{000} 
satisfies the hypothesis in Theorem \ref{theorem4} and Theorem \ref{theorem4b} with $V_2\equiv 0$. Since the $L^2$-norm of the solution 
$w(x,t)$ is preserved for all $t\geq 0$, Theorem \ref{theorem4} and Theorem \ref{theorem4b} yield the desired result.

\end{section}
\begin{section}{Proofs of  Corollaries \ref{corollary11} and Corollary \ref{corollary12} \label{S5}}

\underline{Proof of Corollary \ref{corollary11}.}
\vskip.1in

We observe that if $u(x,t)$ solves the equation in \eqref{e1}, then
\begin{equation}
\label{7l}
w(x,t)=u(x-\mu\,t\,\vec e,t)\,e^{i(\frac{\mu}{2} x\cdot\vec e- \frac{\mu^2 \,t}{4})},
\end{equation}
is a solution of the equation
\begin{equation}
\label{7m}
\partial_t w = i(\Delta w  + V(x-\mu\,t\,\vec e,t)\,w).
\end{equation}
Thus, from hypothesis \eqref{potcon11a} and \eqref{bpotcon11a} the potential in \eqref{7m} 
\begin{equation}
\label{7n}
W(x,t)\equiv V(x-\mu\,t\,\vec e,t)
\end{equation}
satisfies the conditions on Theorems \ref{theorem4} and \ref{theorem4b}, respetively. Therefore, they can be applied to the equation \eqref{7m} to obtain the result.

In the case of the equation \eqref{e1b} the transformation  \eqref{7l} reads
\begin{equation}
\label{7o}
w(x,t)=u(x-\mu\,t\,\vec e,t)\,e^{i(\frac{\mu}{2} x\cdot\vec e(k)-\frac{\mu^2 t\,\mathcal Q_k(\vec e)}{4})},
\end{equation}
with
\begin{equation}
\label{newvec}
\vec e(k)=(e_1,..,e_k,-e_{k+1},..,-e_n),\;\;\;\;\;\text{if}\;\;\;\;\;\vec e=(e_1,...,e_n),
\end{equation}
and $\mathcal Q_k$ as in \eqref{quadform}. The function $w(x,t)$ satisfies the equation
\begin{equation}
\label{7p}
\partial_t w = i(\mathcal L_k w  + V(x-\mu\,t\,\vec e,t)\,w).
\end{equation}
Hence, the potential
\begin{equation}
\label{7q}
W(x,t)\equiv V(x-\mu\,t\,\vec e,t)
\end{equation}
and the solution $w(x,t)$ of \eqref{7p} satisfies the requirements in Theorem \ref{theorem2}. 

\vskip.1in

\underline{Proof of Corollary \ref{corollary12}.}
\vskip.1in
If $u(x,t)$ is a solution of the equation \eqref{f1}
$$
 \partial_t u=i(\Delta u+ F(u,\overline u) \,u),
 $$
 then
 \begin{equation}
\label{7r}
v(x,t)=u(x-\mu\,t\,\vec e,t)\,e^{i(\frac{\mu}{2} x\cdot\vec e-\frac{\mu^2 t}{4})},
\end{equation}
satisfies the equation
\begin{equation}
\label{7s}
\partial_t v = i(\Delta v  + F(e^{-i(\frac{\mu}{2} x\cdot\vec e-\frac{\mu^2 t}{4})}v, e^{i(\frac{\mu}{2} x\cdot\vec e-\frac{\mu^2 t}{4})}\overline v)\, v).
\end{equation}
So in this case from the hypothesis on $F(z,\overline z)$ the potential
\begin{equation}
\label{7t}
W(x,t)\equiv F(e^{-i(\frac{\mu}{2} x\cdot\vec e-\frac{\mu^2 t}{4})}v, e^{i(\frac{\mu}{2} x\cdot\vec e-\frac{\mu^2 t}{4})}\,\overline v),
\end{equation}
verifies that
$$
|W(x,t)|\leq M(|v(x,t)|+|v(x,t)|^j)=M(|u(x-2\mu\,\vec e\,t,t)|+|u(x-2\mu\,\vec e\,t,t)|^j).
$$
Thus, the assumption \eqref{ine} guarantees that we can use Corollary \ref{corollary13} and Theorem \ref{theorem4b} to achieve the result.

In the case of the equation \eqref{f2}
$$
 \partial_t u=i(\mathcal L_k u+ F(u,\overline u) u),
 $$
 one just needs to define $v(x,t)$ as
  \begin{equation}
\label{7v}
v(x,t)=u(x-\mu\,t\,\vec e,t)\,e^{i(\frac{\mu}{2} x\cdot\vec e(k)-\frac{\mu^2 t \mathcal Q_k(\vec e)}{4})},
\end{equation}
with $\vec e(k)$ as in \eqref{newvec} and $\mathcal Q_k$ as in \eqref{quadform}. Since $v(x,t)$ solves the equation
 \begin{equation}
\label{7u}
\partial_t v = i(\mathcal L_k v  + F(e^{-i(\frac{\mu}{2} x\cdot\vec e(k)-\frac{\mu^2 t \mathcal Q_k(\vec e)}{4})}v, 
e^{i(\frac{\mu}{2} x\cdot\vec e(k)-\frac{\mu^2 t \mathcal Q_k(\vec e)}{4})}\overline v) \,v),
\end{equation}
one just needs to follow the argument given in the case of the equation  \eqref{f1} to obtain the desired result.
\end{section}

\begin{section}{Proofs of  Theorem  \ref{theorem20} and Theorem  \ref{theorem20a} \label{S6}}

\underline{Proof of Theorem \ref{theorem20a}.}
\vskip.1in
We have 
$$
e^{\tau |x|} (\Delta +\widetilde V_2) e^{-\tau |x|}=\mathcal S+\mathcal A, 
$$
where
\begin{equation}
\label{defSA}
\mathcal S=\Delta +\widetilde V_2+\tau^2,\;\;\;\;\;\;\; \mathcal A= -\frac\tau{|x|}\left(2x\cdot\nabla+n-1\right).
\end{equation}
Hence, the commutator of $\mathcal S$ and $\mathcal A$ is
\begin{equation*}
[\mathcal S;\mathcal A]=-4 \tau\, \partial_j\cdot((\,\frac{\delta_{jk}}{|x|}-\frac{x_jx_k}{|x|^3})\partial_k\,)+ \frac{ \tau \,(n-1)(n-3)}{|x|^3}+\tau\,\partial_r \widetilde V_2.
\end{equation*}

Let $g\in C_0^\infty(\Rn\setminus \overline {B_\rho})$ and set $f=e^{\tau |x|}g$. Then,
\begin{equation}
\label{E: primera desigualdad}
\begin{aligned}
\|e^{\tau |x|}(\Delta +\widetilde V_2)g\|_2^2&=\|\mathcal Sf\|_2^2+\|\mathcal Af\|_2^2+\int_{\R^n}[\mathcal S;\mathcal A]f\overline f\,dx\\
&=\|\mathcal Sf\|_2^2+\|\mathcal Af\|_2^2+\tau\int_{\R^n} \frac{4}{|x|}\,\left(|\nabla f|^2-|\partial_r f|^2\right)\\
&+\tau\int_{\R^n}\left(\frac{(n-1)(n-3)}{|x|^3}+\partial_r \widetilde V_2\right)|f|^2\,dx,
\end{aligned}
\end{equation}
with $\partial_r f= \frac{x}{|x|} \cdot\nabla f $ and
\begin{equation}
\label{E: segunda desigualdad}
\begin{aligned}
\|\mathcal Af\|_2&= \tau\,\|2\partial_r f+\frac{n-1}{|x|}\,f\|_2\ge \sqrt{\tau}\,\|2\partial_r f+\frac{n-1}{|x|}f\, \|_2\\
&\ge 2\sqrt{\tau}\,\|\partial_r f\|_2-\sqrt\tau\,(n-1)\,\| |x|^{-1} \,f\|_2
\\
&\ge \sqrt{\tau\rho}\,\||x|^{-1/2}\partial_r f\|_2-\sqrt{\tau/\rho}\,\||x|^{-1/2}f\|_2
\end{aligned}
\end{equation}
for $\tau\ge 1$. Combining our hypotheses on the potential \eqref{20a}-\eqref{20c}, \eqref{E: primera desigualdad} and \eqref{E: segunda desigualdad} one gets that 
\begin{equation}\label{E: tercera desiguldad}
\|\mathcal Sf\|_2+\sqrt{\tau\rho}\, \||x|^{-1/2}\nabla f\|_2\le \|e^{\tau |x|}(\Delta +\widetilde V_2) g\|_2+\sqrt{\tau}/\rho\,\||x|^{-1/2}f\|_2.
\end{equation}
Thus using \eqref{defSA} it follows that
\begin{equation}
\label{abc1}
\begin{aligned}
\tau^3&\int_{\R^n} \frac{|f|^2}{|x|}\,dx=\tau\,\Re\int_{\R^n}\frac{1}{|x|}\left[\mathcal Sf\,\overline f-\Delta f\,\overline f-\widetilde V_2|f|^2\right]\,dx\\
\\
&=\tau\,\Re\int_{\R^n}\frac{1}{|x|}\mathcal Sf\,\overline f\,dx-\tau\int_{\Rn}\frac{1}{|x|}\left[\frac 12\Delta |f|^2-|\nabla f|^2+\widetilde V_2|f|^2\right]\,dx\\
\\
&=\tau\,\Re\int_{\R^n}\frac{1}{|x|}\left[\mathcal Sf\,\overline f+|\nabla f|^2+\frac{(n-3)}2 \frac{|f|^2}{|x|^2}-\widetilde V_2|f|^2\right]\,dx.
\end{aligned}
\end{equation}
The last identity, our hypotheses on the potential \eqref{20a}-\eqref{20c}, \eqref{E: tercera desiguldad} and the Cauchy-Schwarz inequality show that Theorem \ref{theorem20a} holds for $\tau\ge\tau_0$ with 
$\tau_0=\tau_0(n, \|\widetilde V\|_{\infty};c_1; c_2;\rho)$.

\vskip.1in

\underline{Proof of Theorem \ref{theorem20}.}
\vskip.1in

We fix $\phi\in C^{\infty}_0(\R^n)$ such that $\phi$ is positive, with $\phi(x)=1,\;|x|\leq 1$ and $supp\,\phi \subset\{x\,:\,|x|\leq2\}$ and rewrite the equation \eqref{eigen} as
\begin{equation}
\label{mia1}
\Delta u +\widetilde V(x) u -\zeta u=\Delta u +\widetilde{\widetilde V}(x)u= \Delta u + \widetilde V_1(x) u + \widetilde V_2(x) u=0,
\end{equation}
with
\begin{equation}
\label{mia2}
\widetilde V_1(x)=\widetilde V(x)-\zeta \phi(x),\;\;\;\;\;\widetilde V_2(x)= -\zeta(1-\phi(x)).
\end{equation}
Thus, $\widetilde V_1,\,\widetilde V_2$ satisfy the hypotheses of Theorems \ref{theorem20} and \ref{theorem20a}.   
We shall define $\phi_L$ as
$$
\phi_L(x)=\phi(x/L),\;\;\;\;\;L>0.
$$
   
 \underline{Claim} :  There exist $\rho_0\in [0,1)$ and $M=M(n)$ such that
 \begin{equation}
 \label{mia3}
 \begin{aligned}
 \|u\|^2_{L^2(B_{4\rho_0})}&=\int_{|x|\leq 4\rho_0} \,|u(x)|^2dx \\
 &\leq M\, \|u\|^2_{L^2(B_{10\rho_0}-B_{5\rho_0})}=\int_{5\rho_0\leq |x|\leq 10\rho_0}|u(x)|^2dx.
 \end{aligned}
 \end{equation}
 
 \underline{Proof of the claim : }
   Multiplying the equation \eqref{mia1} by $u \,\phi^2_{5\rho}$, with $\rho$ to be determined and integrating the result one gets
   \begin{equation}
   \label{mia4}
   -\int |\nabla u|^2\phi^2_{5\rho} \,dx + 
   \int |u|^2 (2 |\nabla \phi_{5\rho}|^2+\phi_{5\rho} \Delta \phi_{5\rho}) \,dx+ \int \Re({\widetilde{\widetilde V}}) |u|^2\phi^2_{5\rho} \,dx.
   \end{equation}
   
  Combining \eqref{mia4} and Poincare inequality one has that
 \begin{equation}
 \label{mia5}
 \begin{aligned}
&  \int |u\phi_{5\rho}|^2dx\leq (10\rho)^2\,\int |\nabla (u\phi_{5\rho}|^2dx\\
&\leq (10\rho)^2\,\int |\nabla u|^2 \phi_{5\rho}^2 dx + c_n \int |u|^2 \phi_{5\rho} |\nabla \phi_{5\rho}|dx\\
&\leq (10\rho)^2( c_n\int_{B_{10\rho}-B_{5\rho}} |u|^2dx +\| \widetilde{\widetilde V}\|_{\infty}\int |u\phi_{5\rho}|^2dx)
+c_n  \int_{B_{10\rho}-B_{5\rho}} |u|^2dx.
\end{aligned}
\end{equation} 
Fixing $\rho_0$ small enough, depending on the $\| \widetilde{\widetilde V}\|_{\infty}$, we establish the claim \eqref{mia3}.

Next, we apply Theorem {theorem20a} to $u\, \Phi=u\,\Phi_{\rho,R}$ where $\Phi\in C^{\infty}_0(\R^n)$ with
$\Phi(x)=1,\,4\rho\leq|x|\leq R$, $\Phi(x)=0,\,|x|\geq 2R$, $\Phi(x)=0,\,|x|\leq 2\rho$ 
with $R>10$ and $\rho\in(0,1)$ to get that
\begin{equation}
\label{mia6}
\begin{aligned}
&\tau^3 \,\| \,|x|^{-1/2} e^{\tau|x|}(u\Phi)\|_2^2\leq \|\,e^{\tau|x|}(\Delta +\widetilde{\widetilde V})(u\Phi)\|_2^2\\
&\leq 4\|\,e^{\tau|x|}\,\nabla u\cdot\nabla\Phi\|_2^2+ 2\|\,e^{\tau|x|}\,u \Delta\Phi)\|_2^2\\
&\leq 4\|\,e^{\tau|x|}\,\nabla u\cdot\nabla\Phi\|_2^2+ 2 c_n \|\,e^{\tau|x|}\,u\|^2_{L^2((B_{2R}-B_R)\cup (B_{4\rho}-B_{2\rho}))}.
\end{aligned}
\end{equation}

   Using integrations by part and the equation \eqref{mia1} one gets that that
 \begin{equation}
 \label{mia7}
   \|\,e^{\tau|x|}\,\nabla u\cdot\nabla\Phi\|_2^2\leq c_n(\|\widetilde{\widetilde V}\|_{\infty}+\tau^2 +\frac{\tau}{\rho}\,)\,\|\,e^{\tau|x|}\,u\cdot\nabla\Phi\|_2^2.
   \end{equation}
   Therefore
$$
A_1\equiv \tau^3 \,\| \,|x|^{-1/2} e^{\tau|x|}(u\Phi)\|_2^2\leq   c_n(\|\widetilde{\widetilde V}\|_{\infty}+\tau^2 +\frac{\tau}{\rho}\,)\,\|\,e^{\tau|x|}\,u\|^2_{L^2((B_{2R}-B_R)\cup (B_{4\rho}-B_{2\rho}))}\equiv A_2.
$$
   
On one hand one has that
$$
A_1\geq \tau^3\,\|\,\frac{e^{\tau|x|}u}{|x|^{1/2}}\|_{L^2(B_R-B_{2\rho})}\geq c_n\,\frac{\tau^3}{\rho}   \,\|e^{\tau|x|}u\|_{L^2(B_{10\rho}-B_{5\rho})}
\geq c_n\,\frac{\tau^3}{\rho} \,e^{10\tau\rho}  \,\|u\|_{L^2(B_{10\rho}-B_{5\rho})}.
 $$  
 On the other hand,
 $$
 A_2\leq c_n(\|\widetilde{\widetilde V}\|_{\infty}+\tau^2 +\frac{\tau}{\rho}\,)\,e^{8\tau\rho}\,\|u\|^2_{L^2(B_{4\rho})} +
 c_n(\|\widetilde{\widetilde V}\|_{\infty}+\tau^2 +\frac{\tau}{\rho}\,)\,e^{4\tau R}\,\|u\|^2_{L^2(B_{2R}-B_R)} .
 $$
 
 Therefore, fixing $\rho=\rho_0$ as in the claim it follows that
 \begin{equation}
 \label{mia9}
 \begin{aligned}
 &M\,\frac{\tau^3}{\rho_0}\,e^{10\tau\,\rho_0} \,\|u\|^2_{L^2(B_{4\rho_0})}\,\leq \frac{\tau^3}{\rho_0} \,e^{10\tau\rho_0}  \,\|u\|_{L^2(B_{10\rho}-B_{5\rho})}\\
 &\leq c_n(\|\widetilde{\widetilde V}\|_{\infty}+\tau^2 +\frac{\tau}{\rho_0}\,)\,e^{8\tau\rho}\,\|u\|^2_{L^2(B_{4\rho)}} +
 c_n(\|\widetilde{\widetilde V}\|_{\infty}+\tau^2 +\frac{\tau}{\rho_0}\,)\,e^{4\tau R}\,\|u\|^2_{L^2(B_{2R}-B_R)} .
\end{aligned}
\end{equation}

 Therefore, for $\tau$ sufficiently large but independently of $R>10$ it follows that
 $$
 \|u\|^2_{L^2(B_{2R}-B_R)}\geq c_n\,e^{10\tau \rho_0}\,e^{-4\tau R}\,\|u\|^2_{L^2(B_{4\rho_0})}.
 $$
 
 Finally, taking $\lambda_0> 2\tau$ one has
 \begin{equation}
 \label{mia07}
 \begin{aligned}
& \infty> \int e^{2\lambda_0|x|} |u(x)|^2dx\geq \sum_{k=1}^{\infty} \int_{2^{k-1}R\leq |x|\leq 2^k} e^{2\lambda_0|x|} |u(x)|^2dx\\
& \geq \sum_{k=1}^{\infty} e^{2^k\lambda_0 R}\,\int_{2^{k-1}R\leq |x|\leq 2^k} |u(x)|^2dx\\
& \geq \sum e^{2^k R\lambda_0} e^{-2^{k+1}\tau R} e^{8\tau \rho_0} \|u\|^2_{L^2(B_{4\rho_0})},
 \end{aligned}
 \end{equation}
 which gives a contradiction  except if $\|u\|^2_{L^2(B_{4\rho_0})}=0$.
\end{section}

\begin{section}{Appendix \label{S7}}

\underbar {Part (a)}: We recall that $p\in(1,4/3]$. The aim is to find 
\begin{equation}
\label{g1}
\varphi(r)=a_0+a_1 r^2+a_2 r^4+ a_3 r^6 + a_4 r^8,\;\;\;\;\;\;r\in[0,1],
\end{equation}
such that
\begin{equation}
\label{g2}
\begin{aligned}
&\varphi(1)=d_0,\;\;\,\varphi'(1)=d_1,\;\;\, \varphi^{(2)}(1)=d_2>0,\\
&\varphi^{(3)}(1)=d_3<0,\;\;\;\varphi^{(4)}(1)=d_4>0.
\end{aligned}
\end{equation}
for prescribed values $d_0,...,d_4$ such that $\varphi(0)=0$ and $\varphi$ is strictly convex for $r\in[0,1]$.
Since in Theorem \ref{theorem4}   $\varphi(r)=r^p+\beta,\;\;r\geq 1$ one has 
\begin{equation}
\label{g3}
\begin{aligned}
&d_0=1+\beta,\;\;\;\;\;d_1=p>0,\;\;\;\;d_2=p(p-1)>0,\\
&d_3=p(p-1)(p-2)<0,\;\;\;\;\;d_4=p(p-1)(p-2)(p-3)>0.
\end{aligned}
\end{equation}

So we solve the system
\begin{equation}
\label{g4}
\begin{aligned}
\begin{cases}
& a_0+a_1+\;\;\;a_2+\;\;\;\;\;a_3+\;\;\;\;\;a_4=d_0=1+\beta,\\
&\;\;\;\;\;\;2a_1+\;4a_2+\;\;\;\;6a_3+\;\;\;\;8a_4=d_1=p,\\
&\;\;\;\;\;\;2a_1+12a_2+\;\;30a_3+\;\;56a_4=d_2=p(p-1),\\
&\;\;\;\;\;\;\;\;\;\;\;\;\;\;\;24 a_2+\;120a_3+\;336a_4=d_3=p(p-1)(p-2),\\
&\;\;\;\;\;\;\;\;\;\;\;\;\;\;\;24 a_2+360a_3+1680a_4=d_4=p(p-1)(p-2)(p-3).
\end{cases}
\end{aligned}
\end{equation}
After some computations one sees that
\begin{equation}
\label{g5}
\begin{aligned}
&a_1=\frac{p}{6\cdot16}(192-104p+18p^2-p^3)>\frac{p}{2},\;\;\;\;a_2=\frac{p(p-2)}{4\cdot16}(p-6)(p-8),\\
&a_3=\frac{-p(p-2)}{6\cdot16}(p-4)(p-8),\;\;\;\;\;\;\;\;\;\;\;\;\;\;\;\;\;\;\;\;\;a_4=\frac{p(p-2)}{24\cdot 16}(p-4)(p-6).
\end{aligned}
\end{equation}
Next, we shall see  that this $\varphi$ is convex in $r\in[0,1]$. From \eqref{g4} and \eqref{g5} one has
\begin{equation}
\label{g6}
\varphi^{(2)}(1)=p,\;\;\;\;\;\;\;\varphi^{(2)}(0)=2a_1>p,
\end{equation}
so it will suffice to show that 
\begin{equation}
\label{g7}
\begin{aligned}
&\varphi^{(3)}(r) =24r (a_2+5a_3r^2+14a_4r^4)\\
&\;=24 r \,\frac{p(p-2)}{12\cdot 16}\left(3 (p-6)(p-8)-10(p-4)(p-8)r^2+7(p-2)(p-6)r^4\right)
\end{aligned}
\end{equation}
has no critical points in $(0,1)$. After some computations one finds that the discriminant $\mathcal D$ of the quadratic equation (in $r^2$) in 
\eqref{g7} is
\begin{equation}
\label{g8}
\begin{aligned}
\mathcal D&=(p-4)(p-8)\left(10^2(p-4)(p-8)-84(p-6)^2\right)\\
&=16(p-1)(p-4)(p-8)(p-11)<0,
\end{aligned}
\end{equation}
because $p\in(1,4/3)$. Since $\varphi^{(3)}$ has no critical points \eqref{g6} tells us that $\varphi$ is strictly convex in $[0,1]$.
Taking $\beta$ in \eqref{g4} as
$$
\beta=a_1+a_2+a_3+a_4-1,
$$
it follows that $\varphi(0)=a_0=0$. Finally, if $ \phi(r)=r^p$ 
$$
\varphi(0)=\varphi'(0)=\phi(0)=\phi'(0)=0,\;\;\;\;\;\;\phi^{(2)}(r)=p(p-1)r^{p-2}\geq p(p-1) \;\;\;r\in(0,1).
$$
Thus, there exists $M_0>0$ such that
$$
M_0 \,p(p-1) \geq \sup_{0\leq r\leq 1} |\varphi^{(2)}(r)|.
$$
Finally, taking $M=max \{M_0;\beta\}$ 
one gets that
$$
\varphi(r)\leq M r^p,\;\;\;\;\;\;\;\forall \;r\geq 0,
$$
which completes the proof.

\vskip.1in

\underbar {Part (b)}: As in the proof of Theorem \ref{theorem4b} we choose
$$
\varphi(r)=3r-\int_1^r\frac{dt}{1+\log t}+\beta,
$$
so in this case we have
\begin{equation}
\label{g3b}
d_0=3+\beta,\;\;\;\;\;\;d_1=2,\;\;\;\;\;d_2=1,\;\;\;\;\;d_3=-3,\;\;\;\;\;\;d_4=14.
\end{equation}
Solving the system \eqref{g4} with these values of $(d_0,d_1,..,d_4)$ one gets
\begin{equation}
\label{g4b}
\varphi(r)=a_0+\frac{103}{96}r^2+\frac{9}{64}r^4-\frac{17}{96}r^6+\frac{17}{24\cdot 16}r^8,\;\;r\in[0,1].
\end{equation}
To show that $\varphi $ is convex in $[0,1]$, we consider
\begin{equation}
\label{g5b}
\varphi^{(2)}(r)=\frac{1}{48}(103+81r^2-225r^4+119r^6),\;\;r\in[0,1],
\end{equation}
and recall that 
\begin{equation}
\label{g10b}
\varphi^{(2)}(0)=103/48,\;\;\;\;\;\varphi^{(2)}(1)=1.
\end{equation}  We look for critical points of 
\begin{equation}
\label{g6b}
\varphi^{(3)}(r)=\frac {r}{8}(27-150 r^2+ 119r^4),\;\;r\in(0,1).
\end{equation}
There is  only one critical point the point $r_0\in (0,1]$ with
\begin{equation}
\label{g7b}
r_0^2=\frac{150-\sqrt{(150)^2-4\cdot119\cdot27}}
{2.119}=\frac{150-\sqrt{9648}}{238}\in(0,1).
\end{equation}
Since
\begin{equation}
\label{g9b}
\varphi^{(2)}(r_0)\geq 110/48,
\end{equation}
combining \eqref{g9b} and \eqref{g10b} it follows that $\varphi $ is convex in $[0,1]$. Finally, taking $\beta$ in \eqref{g3b} such that
$$
\beta=a_1+a_2+a_3+a_4-1,
$$
it follows that $\varphi(0)=a_0=0$. Finally, an argument similar to that at the end of part (a) shows
$$
\varphi(r)\leq M r,\;\;\;\;\;\;\;\forall \;r\geq 0,
$$
which provides the desired result.
\vskip.1in

\vskip.2in
{\underbar{ACKNOWLEDGMENT}}  : The authors would like to thank J. C. Saut for fruitful conversations concerning this work.

\end{section}



\end{document}